\documentclass[11pt]{amsart}

\usepackage{amsfonts,amsmath,amssymb,color,amscd,amsthm}
\usepackage[T1]{fontenc}
\usepackage[francais,english]{babel}
\usepackage{typearea}
\usepackage{color}

\usepackage[T1]{fontenc}

\newtheorem{lemma}{Lemma}[section]

\newtheorem{proposition}[lemma]{Proposition}

\newtheorem*{thm*}{Theorem}

\theoremstyle{definition}
\newtheorem{definition}[lemma]{Definition}
\newtheorem{notation}[lemma]{Notation}
\newtheorem*{SEP}{Stable equivalence problem}

\newtheorem*{GCP}{Generalized cancellation problem}

\theoremstyle{remark}
\newtheorem{remark}[lemma]{Remark}
\newtheorem{example}[lemma]{Example}

\newcommand{\End}{\textup{End}}
\newcommand{\LND}{\textup{LND}}

\newcommand{\Ker}{\textup{Ker}}

\newcommand{\C}{\mathbb{C}}
\newcommand{\N}{\mathbb{N}}

\title{A note on the stable equivalence problem}

\author{Pierre-Marie Poloni}
\address{Universit\"{a}t Basel, Mathematisches Institut, Rheinsprung $21$, CH-$4051$ Basel, Switzerland.}
\email{pierre-marie.poloni@unibas.ch}
\begin{document}
\maketitle
\begin{abstract} We provide counterexamples to the stable equivalence problem in every dimension $d\geq2$. That means that we construct hypersurfaces $H_1, H_2\subset\C^{d+1}$ whose cylinders $H_1\times\C$ and $H_2\times\C$ are equivalent hypersurfaces in $\C^{d+2}$, although $H_1$ and $H_2$ themselves are not equivalent by an automorphism of $\C^{d+1}$. We also give, for every $d\geq2$, examples of two non-isomorphic algebraic varieties of dimension $d$ which are biholomorphic.\end{abstract}

\section{Introduction}

The well known generalized cancellation problem asks the following question.

\begin{GCP} Given two complex affine varieties $V_1$ and $V_2$ with the property that  $V_1\times\C^m$ and $V_2\times\C^m$ are isomorphic for some $m\in\N$. Does this imply that $V_1$ and $V_2$ are isomorphic?\end{GCP}

An affirmative answer  was given by Abhyankar, Eakin and Heinzer \cite{AEH} for the case of affine curves. The cancellation property holds also in the case where $V_1$ (or $V_2$) has nonnegative logarithmic Kodaira dimension. This was shown by  Iitaka and Fujita  in \cite{IF}. However, the answer to the generalized cancellation problem turns out to be negative in general.  The first counterexamples are surfaces due to Danielewski \cite{Danielewski} (see also \cite{Fieseler}). Later on, Danielewski's construction was generalized by  Dubouloz \cite{Dubouloz} to produce counterexamples of every dimension $d\geq2$ (see also \cite{FM} and \cite{DMJP} for factorial and contractible 3-dimensional examples).

In 2004, Makar-Limanov, van Rossum, Shpilrain and  Yu \cite{MLRSY} considered  the following analogous problem.

\begin{SEP} If two hypersurfaces in $\C^n$ are stably equivalent, are they equivalent?
\end{SEP}

Recall that two algebraic varieties $V_1,V_2$ in $\C^n$ are said to be \emph{equivalent} if there exists a polynomial automorphism of $\C^n$ which maps $V_1$ onto $V_2$, and that they are said to be \emph{stably equivalent} if there is an integer $m\in\N$ such that the cylinders $V_1\times\C^m$ and $V_2\times\C^m$ are equivalent varieties in $\C^{n+m}$.
The stable equivalent  problem has a positive answer for affine plane curves, as already shown by Makar-Limanov, van Rossum, Shpilrain and  Yu in \cite{MLRSY}. In the same vein of the result of Iitaka-Fujita,  Drylo proved in \cite{Drylo} that two stably equivalent hypersurfaces in $\C^n$ are equivalent, if one of them is not $\C$-uniruled. The first counterexamples in $\C^3$, consisting in families of Danielewski hypersurfaces, were provided by Moser-Jauslin and the author \cite{MJP}. Also,  contractible 3-dimensional counterexamples appeared in \cite{DMJP}.

In this note, we complete the analogy between the results on the generalized cancellation and stable equivalence problems. Indeed, we produce counterexamples to the stable equivalence problem for every $n\geq3$. These new examples are easy generalizations  of those of \cite{MJP}, inspired by the construction in \cite{Dubouloz}.

We will actually give two kinds of counterexamples. On one hand, polynomials $P,Q\in\C[X_1,\ldots,X_n]$ whose zero-sets $V(P)$ and $V(Q)$ are non-isomorphic varieties, but such that the cylinders $V(P)\times\C$ and $V(Q)\times\C$ are equivalent hypersurfaces in $\C^{n+1}$. On the other hand, polynomials $P,Q\in\C[X_1,\ldots,X_n]$ with the properties that  $V(P)\times\C$ and $V(Q)\times\C$ are equivalent hypersurfaces in $\C^{n+1}$ and  that $V(P)$ and $V(Q)$ are non-equivalent hypersurfaces in $\C^n$, although the fibers $V(P-c)$ and $V(Q-c)$ of $P$ and $Q$ are pairwise isomorphic for all $c\in\C$. More precisely, we will prove the following result.

\begin{thm*}The following assertions hold for every natural number $n\geq1$.
\begin{enumerate}
\item The hypersurfaces $H_1, H_2\subset\C^{n+2}$  defined by the equation  $x_1^2\cdots x_n^2y+z^2+x_1\cdots x_n(z^2-1)=1$ and  $x_1^2\cdots x_n^2y+z^2+x_1\cdots x_n(z^2-2)=1$, respectively, are non-isomorphic algebraic varieties such that  $H_1\times\C$ and $H_2\times\C$ are equivalent hypersurfaces in $\C^{n+3}$.
\item The polynomials $Q_k=x_1^2\cdots x_n^2y+z^2+x_1\cdots x_n(z^2-1)^k\in\C[x_1,\ldots,x_n,y,z]$ are stably equivalent for all $k\geq1$, whereas the hypersurfaces $V(Q_k)\subset\C^{n+2}$ are pairwise non-equivalent. However, the varieties $V(Q_k-c)$ and $V(Q_{k'}-c)$ are isomorphic for all $k,k'\geq1$ and every $c\in\C$.
\end{enumerate}
\end{thm*}

It is  worth mentioning that the special case of affine spaces is still open, for both cancellation and stable equivalence problems. Recall that the question to know whether an isomorphism $V\times\C^m\simeq\C^{n+m}$ implies $V\simeq\C^n$ is usually referred to as the ``Zariski cancellation problem''. It has a positive solution for $n=1$ and for $n=2$ by the results of Fujita and Miyanishi-Sugie (\cite{Fujita}, \cite{MiySugie}), whereas it is still an unsolved problem for $n\geq3$.

Similarly, it was asked in \cite{MLRSY} if every hypersurface in $\C^{n+1}$, which is stably equivalent to a (linear) hyperplane, is already equivalent to this hyperplane. Note that it is true for $n=1$ and also, using the cancellation property of the affine plane and a result of Kaliman \cite{Kaliman}, for $n=2$. Moreover, as noticed in \cite{MLRSY}, a positive answer to this question for an integer $n\geq3$ would imply that the $n$-dimensional affine space has the cancellation property.

\section{Four hypersurfaces in $\C^{n+2}$}

Let us  fix some notations.

\begin{notation}\label{notation_Pq}
Given a ring $R$ and an integer $m\in\N$, we denote by $R^{[m]}$ the polynomial ring in $m$ variables over $R$.
Throughout this paper,  we fix  a positive integer $n$ and we  denote by $\C[\underline{x}]$ the polynomial ring $\C[x_1,\ldots,x_n]\simeq\C^{[n]}$ in the variables $x_1,\ldots,x_n$.

For every integer $k\in\N$, we  denote by $\underline{x}^{[k]}$ the element $\underline{x}^{[k]}=x_1^k\cdots x_n^k\in\C[\underline{x}]$ and, for every polynomial $q\in\C^{[1]}$,  by $P_q$ the polynomial of $\C[x_1,\ldots,x_n,y,z]=\C[\underline{x}][y,z]$ defined by $$P_q=\underline{x}^{[2]}y+z^2+\underline{x}^{[1]}q(z^2).$$
\end{notation}

The  counterexamples to the stable equivalent problem mentioned in the introduction are realized as hypersurfaces in $\C^{n+2}$ given by the fibers $V(P_q-c)$ of some polynomials $P_q$.  We will  determine the isomorphism classes of these varieties. This will be done by using techniques mainly developed by Makar-Limanov  in  \cite{ML01}.  The idea is to exploit the fact that  this kind of hypersurfaces admit  additive group actions, but not too many. For instance, their Makar-Limanov invariants are non trivial.

It is in general very difficult and technical to compute such invariants. But we are in a good situation, since the method of Kaliman and Makar-Limanov (\cite{Kaliman-ML}) applies to the varieties that we are considering. Moreover, we can even use directly  the results of Dubouloz \cite{Dubouloz}, who already did the computation for the  case where the polynomial $q$ is constant. Remark that, thanks to the next lemma,  it suffices to consider only this  special case.

\begin{lemma}\label{lem_phic}
Let $R=\C[\underline{x}]\simeq\C^{[n]}$. Given $q\in\C^{[1]}$  and $c\in\C$, we let $g_c\in\C^{[1]}$ be the polynomial such that the equality $q(z^2)-q(c)=g_c(z^2)(z^2-c)$ holds in $\C[z]$. Then, the endomorphism $\varphi_c\in\End_{R}R[y,z]$ of $R[y,z]$ fixing $R$ and defined by $$\varphi_c(y)=\left(1+\underline{x}^{[1]}g_c(z^2)\right)y+q(c)g_c(z^2)\quad\text{and}\quad\varphi_c(z)=z$$
induces an isomorphism between the rings $\C[\underline{x},y,z]/(P_q-c)$ and $\C[\underline{x},y,z]/(P_{q(c)}-c)$.
\end{lemma}

\begin{proof}
First, one checks that $\varphi_c\left(P_q-c\right)=\left(1+\underline{x}^{[1]}g_c(z^2)\right)\left(P_{q(c)}-c\right)$. Thus, $\varphi_c$ induces a morphism between  $\C[\underline{x},y,z]/(P_q-c)$ and $\C[\underline{x},y,z]/(P_{q(c)}-c)$. The latter is invertible. To see this, one checks that the inverse morphism is induced by the endomorphism $\psi_c\in\End_{R}R[y,z]$ defined by $\psi_c(y)=\left(1-\underline{x}^{[1]}g_c(z^2)\right)y-q(z^2)g_c(z^2)$ and $\varphi_c(z)=z$.
\end{proof}

We will now compute, for all $q\in\C^{[1]}$ and all $c\in\C$, the set $\LND(B_{q,c})$ of locally nilpotent derivations on the coordinate ring $B_{q,c}$ of the varieties $V(P_q-c)$. Recall that a derivation $\delta$ of a $\C$-algebra $B$ is called \emph{locally nilpotent} if there exists, for every element $b\in B$, an integer $m=m(b)\geq1$ such that $\delta^m(b)=0$.
Let $\Delta$ be the derivation of $\C[\underline{x},y,z]$ defined by $$\Delta=\underline{x}^{[2]}\frac{\partial}{\partial z}-2z(1+\underline{x}^{[1]}q'(z^2))\frac{\partial}{\partial y},$$
where $q'$ denotes the derivative of $q$. Note that $\Delta$ is locally nilpotent (it is a triangular derivation) and that it annihilates the polynomial $P_q-c$. Therefore, it induces a locally nilpotent derivation on $B_{q,c}$, which we still denote by $\Delta$. It turns out that all other locally nilpotent derivations on $B_{q,c}$ are multiple of $\Delta$ by elements of $\C[\underline{x}]$.

\begin{proposition}\label{prop-lnd}
Let $q\in\C^{[1]}$, $c\in\C$ and $B_{q,c}=\C[\underline{x},y,z]/(P_q-c)$, where $P_q=\underline{x}^{[2]}y+z^2+\underline{x}^{[1]}q(z^2)\in\C[\underline{x},y,z]$.  Then, the following hold for every nonzero locally nilpotent derivation $\delta$ of $B_{q,c}$.
\begin{enumerate}
\item $\Ker(\delta)=\C[\underline{x}]$  and $\Ker(\delta^2)=\C[\underline{x}]z+\C[\underline{x}]$.
\item There exists $h(\underline{x})\in\C[\underline{x}]$ such that $\delta=h(\underline{x})\Delta$,
where $\Delta$ is the locally nilpotent derivation on $B_{q,c}$ defined above.
\end{enumerate}
\end{proposition}

\begin{proof}
(1) First of all, remark that we can suppose that $q$ is a constant polynomial.  Indeed, take the isomorphism $\phi:B_{q,c}\to B_{q(c),c}$ given by Lemma \ref{lem_phic} and let $\delta\in\LND(B_{q,c})\setminus\{0\}$ be a nonzero locally nilpotent derivation. Then, $\widetilde\delta=\phi\circ\delta\circ\phi^{-1}\in\LND(B_{q(c),c})\setminus\{0\}$ and we have $\Ker(\delta)=\phi^{-1}(\Ker(\widetilde\delta))$ and $\Ker(\delta^2)=\phi^{-1}(\Ker(\widetilde\delta^2))$. Since $\phi^{-1}$ maps $\C[\underline{x}]$ onto $\C[\underline{x}]$ and $\C[\underline{x}]z+\C[\underline{x}]$ onto $\C[\underline{x}]z+\C[\underline{x}]$, it suffices to prove $\Ker(\widetilde\delta)=\C[\underline{x}]$ and $\Ker(\widetilde\delta^2)=\C[\underline{x}]z+\C[\underline{x}]$.

So, let $q,c\in\C$ be two constants and let $\delta$ be a nonzero locally nilpotent derivation on $B_{q,c}$. We are now in the case considered by  Dubouloz  in \cite{Dubouloz}, where he proved (see paragraph 2.7 in \cite{Dubouloz}) that $\Ker(\delta)=\C[\underline{x}]$  and $\Ker(\delta^2)\subset\C[\underline{x},z]$ hold. This implies easily $\Ker(\delta^2)=\C[\underline{x}]z+\C[\underline{x}]$.

Indeed, let $a\in\Ker(\delta^2)\setminus\Ker(\delta)$ and write $a=\sum_{i=0}^{d}\alpha_i(\underline{x})z^i$ with $d\geq1$ and $\alpha_i(\underline{x})\in\C[\underline{x}]$. Then $\delta(a)=\delta(z)\sum_{i=1}^{d}i\alpha_i(\underline{x})z^{i-1}$ is a nonzero element of $\Ker(\delta)$. Since the kernel of a locally nilpotent derivation is factorially closed, it follows that $\delta(z)$ lies in $\Ker(\delta)$. Thus, $z\in\Ker(\delta^2)$ and so $\C[\underline{x}]z+\C[\underline{x}]\subset\Ker(\delta^2)$. On the other hand, $\delta(a)\in\Ker(\delta)$ implies $d=1$, since $\Ker(\delta)=\C[\underline{x}]$. Therefore, $a\in\C[\underline{x}]z+\C[\underline{x}]$ and (1) is proved.

(2) Let $\delta\in\LND(B_{q,c})\setminus\{0\}$. By (1), $\Ker(\delta)=\C[\underline{x}]$ and there exists a polynomial $a(\underline{x})\in\C[\underline{x}]\setminus\{0\}$  such that $\delta(z)=a(\underline{x})$. To prove (2), it suffices to find an element $h(\underline{x})\in\C[\underline{x}]$ such that $a(\underline{x})=\underline{x}^{[2]}h(\underline{x})$, since $$0=\delta(P_q-c)=\underline{x}^{[2]}\delta(y)+a(\underline{x})2z(1+\underline{x}^{[1]}q'(z^2)).$$
The equality above means that there exist  polynomials $F,R\in\C^{[n+2]}$ such that
$$\underline{X}^{[2]}F(\underline{X},Y,Z)+a(\underline{X})2Z(1+\underline{X}^{[1]}q'(Z^2))=R(\underline{X},Y,Z)(\underline{X}^{[2]}Y+Z^2+\underline{X}^{[1]}q(Z^2)-c).$$
From this, it follows that $a(\underline{X})$ and $R(\underline{X},Y,Z)$ are both divisible by $\underline{X}^{[1]}$. Setting $a(\underline{X})=\underline{X}^{[1]}\widetilde{a}(\underline{X})$ and $R(\underline{X},Y,Z)=\underline{X}^{[1]}\widetilde{R}(\underline{X},Y,Z)$, we obtain the equality
$$\underline{X}^{[1]}F(\underline{X},Y,Z)+\widetilde{a}(\underline{X})2Z(1+\underline{X}^{[1]}q'(Z^2))=\widetilde{R}(\underline{X},Y,Z)(\underline{X}^{[2]}Y+Z^2+\underline{X}^{[1]}q(Z^2)-c).$$
The latter  implies that $\widetilde{a}(\underline{X})$ is divisible by $\underline{X}^{[1]}$. Thus, $a(\underline{x})=\underline{x}^{[2]}h(\underline{x})$ for an element $h(\underline{x})\in\C[\underline{x}]$. This completes the proof.
\end{proof}

We are now in position to classify all hypersurfaces in $\C^{n+2}$ given by an equation of the form $P_q=c$. They have exactly four isomorphism classes. Each of them is given by one of the following varieties.

\begin{notation} We denote by $V_{0,0}, V_{0,1}, V_{1,0}, V_{1,1}$ the hypersurfaces in $\C^{n+2}$ defined by the equation $\underline{x}^{[2]}y+z^2=0$, $\underline{x}^{[2]}y+z^2-1=0$, $\underline{x}^{[2]}y+z^2+\underline{x}^{[1]}=0$ and $\underline{x}^{[2]}y+z^2+\underline{x}^{[1]}-1=0$, respectively.
\end{notation}

These varieties are pairwise non-isomorphic and we have the following result, which was already proved in \cite{MJP} for the case $n=1$.

\begin{proposition}\label{prop_isom}
Let $q\in\C^{[1]}$, $c\in\C$. and let $P_q=\underline{x}^{[2]}y+z^2+\underline{x}^{[1]}q(z^2)\in\C[\underline{x},y,z]$ as in Notation \ref{notation_Pq}. Then, the variety $V(P_q-c)$ is isomorphic to:
\begin{enumerate}
\item  $V_{0,0}$ if an only if $c=0$ and $q(c)=0$;
\item  $V_{1,0}$ if an only if $c=0$ and $q(c)\neq0$;
\item  $V_{0,1}$ if an only if $c\neq0$ and $q(c)=0$;
\item  $V_{1,1}$ if an only if $c\neq0$ and $q(c)\neq0$.
\end{enumerate}
\end{proposition}

\begin{proof}
By Lemma \ref{lem_phic}, the variety $V(P_q-c)$ is isomorphic to the hypersurface of equation $$\underline{x}^{[2]}y+z^2+\underline{x}^{[1]}q(c)-c=0.$$ The ``if parts'' of the proposition follow then easily.

In order to prove that  $V_{0,0}, V_{0,1}, V_{1,0}$ and $V_{1,1}$ are non-isomorphic, we consider two polynomials  $q_1,q_2\in\C^{[1]}$ and two constants $c_1,c_2\in\C$. For $j=1,2$, let $B_j$ denotes the ring $B_j=\C[\underline{x},y,z]/(P_{q_j}-c_j)$ and let  $x_{i_j},y_j,z_j$ denote the images of $x_i, y, z$ in $B_j$. We also denote by $\C\left[\underline{x}_j\right]$ the ring $\C[x_{1_j},\ldots,x_{n_j}]$. Suppose now that $\varphi:B_1\to B_2$ is an isomorphism.

Let $\delta\in\LND(B_1)\setminus\{0\}$ be a nonzero locally derivation on $B_1$. Then, $\widetilde\delta=\varphi\circ\delta\circ\varphi^{-1}$ is a nonzero locally derivation on $B_2$ and we have $\Ker(\widetilde\delta)=\varphi(\Ker(\delta))$ and $\Ker((\widetilde\delta)^2)=\varphi(\Ker(\delta^2))$. By Proposition \ref{prop-lnd}, we have $\Ker(\delta)=\C[\underline{x}_1]$ and $\Ker(\widetilde\delta)=\C[\underline{x}_2]$. Thus,  $\varphi$ restricts to an isomorphism between $\C[\underline{x}_1]$ and $\C[\underline{x}_2]$. Moreover $\varphi(z_1)\in\Ker((\widetilde\delta)^2)=\C[\underline{x}_2]z_2+\C[\underline{x}_2]$. Therefore, $\varphi(z_1)=\alpha(\underline{x}_2)z_2+\beta(\underline{x}_2)$ for some polynomials $\alpha$ and $\beta$. Repeating the same argument with $\varphi^{-1}$, we obtain that $\varphi^{-1}(z_2)=a(\underline{x}_1)z_1+b(\underline{x}_1)$ for some polynomials $a$ and $b$. From this, we get that the elements $\alpha(\underline{x}_2)\in\C[\underline{x}_2]$ and $a(\underline{x}_1)\in\C[\underline{x}_1]$
are in fact invertible, thus nonzero constants.

If we take  the derivation $\delta=\Delta$ (see Proposition \ref{prop-lnd}), one checks  that $\widetilde\delta(z_2)=\varphi(\Delta(a z_1+b(\underline{x}_1)))=a\varphi(\underline{x}_1^{[2]})$. Consequently,  there exists, again by Proposition \ref{prop-lnd}, a polynomial $h$ such that $a\varphi(\underline{x}_1^{[2]})=h(\underline{x}_2)\underline{x}_2^{[2]}$. Since $\varphi:\C[\underline{x}_1]\to\C[\underline{x}_2]$ is an isomorphism, this implies that there exist a bijection $\sigma$ of the set $\{1,\ldots,n\}$  and nonzero constants $\lambda_{i}\in\C^*$ such that $\varphi(x_i)=\lambda_{i}x_{\sigma(i)}$ for all $1\leq i\leq n$.

Let $\lambda=\prod_{i=1}^n\lambda_i$ and suppose from now on that $q_1$ and $q_2$ are constant. Since $\lambda^2\underline{x}_2^{[2]}\varphi(y_1)+(\alpha z_2+\beta(\underline{x}_2))^2+\lambda\underline{x}_2^{[1]}q_1-c_1=\varphi(\underline{x}_1^{[2]}y_1+z_1^2+\underline{x}_1^{[1]}q_1-c_1)=0$ in $B_2$, there exist polynomials $F,A\in\C^{[n+2]}$ such that
$$\lambda^2\underline{x}^{[2]}F(\underline{x},y,z)+(\alpha z+\beta(\underline{x}))^2+\lambda q_1\underline{x}^{[1]}-c_1=A(\underline{x},y,z)(\underline{x}^{[2]}y+z^2+q_2\underline{x}^{[1]}-c_2). $$
Looking at this equality modulo $(\underline{x}^{[2]})$, it follows that $\beta(\underline{x})$ lies in the ideal  of $\C[\underline{x}]$ generated by $\underline{x}^{[2]}$, and that $c_1=A(\underline{0},0,0)c_2$ and $\lambda q_1=A(\underline{0},0,0)q_2$. This shows that $V_{0,0}, V_{0,1}, V_{1,0}, V_{1,1}$ are pairwise non-isomorphic and  proves the proposition.
\end{proof}

\begin{remark}
Even if they are non-isomorphic,  the varieties $V_{0,1}$ and $V_{1,1}$ are biholomorphic. Indeed, the analytic automorphism $\Psi$  of $\C[\underline{x},y,z]$ defined by $\Psi(x_i)=x_i$ for all $1\leq i\leq n$, $$\Psi(y)=\exp(-\underline{x}^{[1]})y-\frac{\exp(-\underline{x}^{[1]})-1+\underline{x}^{[1]}}{\underline{x}^{[2]}}\quad\text{and}\quad \Psi(z)=\exp(-\frac{1}{2}\underline{x}^{[1]})z,$$ satisfies $\Psi(\underline{x}^{[2]}y+z^2+\underline{x}^{[1]}-1)=\exp(-\underline{x}^{[1]})(\underline{x}^{[2]}y+z^2-1)$.
The case $n=1$ is due to Freudenburg and Moser-Jauslin \cite{Freudenburg-MJ} and it was, to our knowledge, the first explicit example in the literature of two algebraically non-isomorphic varieties that are holomorphically isomorphic. Note that Jelonek \cite{Jelonek} has recently constructed other examples, in every dimension $d\geq2$, of rational varieties  with these properties.

\end{remark}

\section{Stable equivalence}

In this paper, we will consider  two  notions of equivalence.

\begin{definition}\hfill
\begin{enumerate}
\item Two hypersurfaces $H_1,H_2\subset\C^n$ are said to be \emph{equivalent} if there exists a polynomial automorphism $\Phi$ of $\C^n$ such that $\Phi(H_1)=H_2$.
\item Two polynomials $P,Q\in\C^{[n]}$ are said to be \emph{equivalent} if there exists a polynomial automorphism $\Phi$ of $\C^n$ such that $\Phi^*(P)=Q$.
\end{enumerate}
\end{definition}

These two notions are of course closely related, the zero-sets $V(P)$ and $V(Q)$ of irreducible polynomials  $P,Q\in\C^{[n]}$ being equivalent hypersurfaces in $\C^n$ if and only if there exists a nonzero constant $\mu\in\C^*$ such that $P$ and $\mu Q$ are equivalent polynomials in $\C^{[n]}$.

The next proposition gives the classification, up to equivalence, of all polynomials $P_q$ (see Notation \ref{notation_Pq}) and of their fibers $V(P_q-c)$. It is an easy generalization  of results of \cite{MJP} to the case $n\geq2$.

\begin{proposition}\label{prop-equiv}
Let $q_1,q_2\in\C^{[1]}$ be two polynomials and $c_1,c_2\in\C$ be two constants. For $i=1,2$, let $P_{q_i}=\underline{x}^{[2]}y+z^2+\underline{x}^{[1]}q_i(z^2)\in\C[\underline{x},y,z]$ as in Notation \ref{notation_Pq}. Then, the following hold.
\begin{enumerate}
\item  The polynomials $P_{q_1}-c_1$ and $P_{q_2}-c_2$ of $\C^{[n+2]}$ are equivalent if and only if $c_1=c_2$ and there exists a nonzero constant $\lambda\in\C^*$ such that $q_2=\lambda q_1$.

\item The hypersurfaces $H_1=V(P_{q_1}-c_1), H_2=V(P_{q_2}-c_2)\subset\C^{n+2}$ are equivalent if and only if there exist two nonzero constants $\lambda,\mu\in\C^*$ such that $c_2=\mu^{-1}c_1$ and  such that the equality $q_2(t)=\lambda q_1(\mu t)$ holds in $\C[t]$.
\end{enumerate}
\end{proposition}

\begin{proof}
(1) Suppose that $P_{q_1}-c_1$ and  $P_{q_2}-c_2$ are equivalent polynomials of $\C[\underline{x},y,z]$ and let $\Phi$ be an automorphism of $\C[\underline{x},y,z]$ such that $\Phi(P_{q_1}-c_1)=P_{q_2}-c_2$. The key of the proof is to show that $\Phi(\underline{x}^{[1]})=\lambda\underline{x}^{[1]}$ for some constant $\lambda\in\C^*$. Afterwards, we can conclude exactly as in \cite{MJP}.

Remark that $\Phi$ induces, for every $c\in\C$, an isomorphism $\Phi_c$ between the rings $B_1=\C[\underline{x},y,z]/(P_{q_1}-c_1-c)$ and $B_2=\C[\underline{x},y,z]/(P_{q_2}-c_2-c)$. Therefore, as we have seen in the proof of Proposition \ref{prop_isom}, the element $\Phi_c(\underline{x}^{[1]})$ lies in the ideal $\underline{x}^{[1]}B_2$. Thus, $$\Phi(\underline{x}^{[1]})\in\bigcap_{c\in\C}\left(\underline{x}^{[1]}, P_{q_2}-c_2-c\right)=\bigcap_{c\in\C}\left(\underline{x}^{[1]},z^2-c_2-c\right)=\left(\underline{x}^{[1]}\right).$$
Since $\Phi$ is an automorphism, this implies that there exists a nonzero constant $\lambda\in\C^*$ such that $\Phi(\underline{x}^{[1]})=\lambda\underline{x}^{[1]}$, as desired.

Now, since
$\Phi(P_{q_1}-c_1+\alpha\underline{x}^{[1]}-c)=P_{q_2}-c_2+\alpha\lambda\underline{x}^{[1]}-c$, the varieties $V(P_{q_1+\alpha}-c_1-c)$ and $V(P_{q_2+\alpha\lambda}-c_2-c)$ are isomorphic for all $\alpha,c\in\C$. By Proposition \ref{prop_isom}, this implies that $c_1=c_2$ and then that the zeros of the polynomials $q_1+\alpha$ and $q_2+\alpha\lambda$ are the same for all $\alpha\in\C$. Thus, $q_2=\lambda q_1$.

Conversely, if $q_2=\lambda q_1$ for some $\lambda\in\C^*$, it suffices to check that $\Phi(P_{q_1})=P_{q_2}$, where  $\Phi$ is the  automorphism of $\C[\underline{x},y,z]$ defined by $\Phi(x_1)=\lambda x_1$, $\Phi(x_i)=x_i$ for all $2\leq i\leq n$, $\Phi(y)=\lambda^{-2}y$ and $\Phi(z)=z$. This proves the assertion $(1)$.

(2) The hypersurfaces $H_1=V(P_{q_1}-c_1)$ and $H_2=V(P_{q_2}-c_2)$ are equivalent if and only if there exists a nonzero constant $\mu\in\C^*$ such that the polynomials $P_{q_1}-c_1$ and $\mu(P_{q_2}-c_2)$ are equivalent. Then, Assertion (2) follows from Assertion (1), noting that  $\mu(P_{q_2}-c_2)$ is equivalent to the polynomial  $P_{\widetilde{q_2}}-\mu c_2$, where $\widetilde{q_2}$ denotes the element  of $\C[t]$ defined by $\widetilde{q_2}(t)=q_2(\mu^{-1}t)$. Indeed, one checks that this equivalence is realized by the automorphism of $\C^{n+2}$ defined by $(x_1,x_2,\ldots,x_n,y,z)\mapsto(\mu^{-1}x_1,x_2,\ldots,x_n,\mu y,\epsilon z)$, where $\epsilon$ is any complex number such that $\epsilon^2=\mu^{-1}$.
\end{proof}

Before we state the next result, let us recall the notion of \emph{stable equivalence}.

\begin{definition}\hfill
\begin{enumerate}
\item Two hypersurfaces $H_1,H_2\subset\C^n$ are said to be \emph{stably equivalent} if there exists a $m\in\N$ such that $H_1\times\C^m$ and $H_2\times\C^m$ are equivalent hypersurfaces in $\C^{n+m}$.
\item Two polynomials $P,Q\in\C^{[n]}$ are said to be \emph{stably equivalent} if there exists a $m\in\N$ such that $P$ and $Q$ are equivalent polynomials of $\C^{[n+m]}$.
\end{enumerate}
\end{definition}

In this context, we have the following obvious generalization of Theorem 2.5' of \cite{MJP}.

\begin{lemma}\label{lemme-equiv-stable}
For every  $q\in\C^{[1]}$, the polynomials $P_q$ and $P_{q(0)}$ are stably equivalent.
\end{lemma}

\begin{proof}
The case $n=1$ was proved in \cite{MJP}, where an explicit automorphism $\Phi$ of $\C[x,y,z,w]$, fixing $x$ and satisfying $\Phi(x^2y+z^2+xq(z^2))=x^2y+z^2+xq(0)$, is constructed. Since this automorphism fixes $x$, it suffices to replace formally $x$ by $\underline{x}^{[1]}$ to get an automorphism of $\C[\underline{x},y,z,w]$ which maps $P_q$ onto $P_{q(0)}$.   For the sake of completeness, let us give the formula.

Let $r\in\C[t]$ be the polynomial such that the equality $q(t)-q(0)=2t r(t)$ holds. We let $\Phi(x_i)=x_i$ for all $1\leq i\leq n$, $\Phi(z)=(1-\underline{x}^{[1]}r(P_{q(0)}))z+\underline{x}^{[2]}w$ and $\Phi(w)=(1+\underline{x}^{[1]}r(P_{q(0)}))w-(r(P_{q(0)}))^2z$. Note that $\Phi(z^2+\underline{x}^{[1]}q(z^2))\equiv z^2+\underline{x}^{[1]}q(0)\mod(\underline{x}^{[2]})$. Therefore, we can choose $\Phi(y)\in\C[\underline{x},y,z,w]$ such that $\Phi(P_q)=P_{q(0)}$. Doing so, we get an endomorphism (we will show that it is in fact an automorphism) $\Phi$ of  $\C[\underline{x},y,z,w]$ which maps $P_q$ onto $P_{q(0)}$.

Similarly, we define an endomorphism $\Psi$ of  $\C[\underline{x},y,z,w]$ such that  $\Psi(P_{q(0)})=P_q$ by posing $\Psi(x_i)=x_i$ for all $1\leq i\leq n$, $\Psi(z)=(1+\underline{x}^{[1]}r(P_{q}))z-\underline{x}^{[2]}w$ and $\Psi(w)=(1-\underline{x}^{[1]}r(P_{q}))w+(r(P_{q}))^2z$.

Now, one checks that $\Phi\circ\Psi(z)=z$ and that $\Phi\circ\Psi(w)=w$. Moreover, since $\Phi\circ\Psi(P_{q(0)})=P_{q(0)}$, we have $\underline{x}^{[2]}\Phi\circ\Psi(y)+z^2+\underline{x}^{[1]}q(0)=\Phi\circ\Psi(P_{q(0)})=P_{q(0)}=\underline{x}^{[2]}y+z^2+\underline{x}^{[1]}q(0)$. This implies that  $\Phi\circ\Psi(y)=y$. Therefore, $\Psi$ is the inverse morphism of $\Phi$. This proves the lemma.
\end{proof}

Together with Propositions \ref{prop_isom} and \ref{prop-equiv}, Lemma \ref{lemme-equiv-stable} leads to many counterexamples to the ``stable equivalence problem'' of every dimension $d\geq2$. Finally, let us emphasize two particular examples.

\begin{example}\hfill
\begin{enumerate}
\item The polynomials $P=\underline{x}^{[2]}y+z^2+\underline{x}^{[1]}(z^2-1)-1$ and $Q=\underline{x}^{[2]}y+z^2+\underline{x}^{[1]}(z^2-1)-1$ of $\C[\underline{x},y,z]$ are stably equivalent, but the hypersurfaces $V(P)$ and $V(Q)$ in $\C^{n+2}$ are not equivalent. Indeed, they are even non-isomorphic varieties.
\item The polynomials $Q_k=\underline{x}^{[2]}y+z^2+\underline{x}^{[1]}(z^2-1)^k\in\C[\underline{x},y,z]$ are stably equivalent for all $k\geq1$, whereas the hypersurfaces $V(Q_k)\subset\C^{n+2}$ are pairwise non-equivalent. However, the varieties $V(Q_k-c)$ and $V(Q_{k'}-c)$ are isomorphic for all $k,k'\geq1$ and every $c\in\C$.
\end{enumerate}
\end{example}

\end{document}